\documentclass[12pt]{amsart}
\usepackage[T1]{fontenc}
\usepackage{lmodern}
\usepackage[margin=1in]{geometry}
\usepackage{amsmath,amssymb,amsthm,mathtools}
\usepackage{microtype}
\usepackage[colorlinks=true,linkcolor=black,citecolor=black,urlcolor=blue]{hyperref}
\newtheorem{theorem}{Theorem}[section]
\newtheorem{corollary}{Corollary}[section]
\newtheorem{example}{Example}[section]
\newtheorem{remark}{Remark}[section]
\newtheorem{lemma}{Lemma}[section]
\newtheorem{proposition}{Proposition}[section]
\newcommand{\R}{\mathbb R}
\newcommand{\Ric}{\operatorname{Ric}}

\newcommand{\dd}{\,d}
\numberwithin{equation}{section}
\title[Complete self-shrinkers with positive scalar curvature]{Complete self-shrinkers with \\ positive scalar curvature*}
\author [Q.-M. Cheng, F. Li, and G. Wei]{Qing-Ming Cheng, Fengjiang Li{$^+$}, and Guoxin Wei{$^+$}}
\address{Qing-Ming Cheng \newline
\indent
Mathematical Science Research Center, \newline 
\indent Chongqing University of Technology, Chongqing 400054, P. R. China. 
\vskip 1mm
\indent
Osaka Center Advanced Mathematical Institute \newline
\indent
Osaka Metropolitan University, Osaka 558-8585, Japan \newline
\indent qingmingcheng@yahoo.com, chengqingming@cqut.edu.cn }
\address{Fengjiang Li \newline
\indent
Mathematical Science Research Center, \newline 
\indent Chongqing University of Technology, Chongqing 400054, P. R. China. }
\address{Guoxin Wei \newline
\indent   School of Mathematical Sciences, South China Normal University,
\newline
\indent 510631, Guangzhou,  China, weiguoxin@tsinghua.org.cn}

\thanks{$+$ Fengjiang Li and Guoxin Wei are corresponding authors}
\thanks{*This work was supported by  National Natural Science Foundation of China (Grant Nos.
12301062, 12571050, 12671065), Natural Science Foundation of Chongqing
(No.CSTB2024NSCQ-MSX0537), Japan Society for the Promotion of Science Grant-in-Aid
for Scientific Research (C) (Grant No. 25K06992).}


\begin{document}
\begin{abstract}
Let $X:M^n\to\R^{n+1}$ be an $n$-dimensional  complete  self-shrinker. 
We obtain  a complete classification of complete self-shrinkers with positive  constant scalar curvature. 
The key difficulty is to characterize the residual case in Cheng--Li--Wei \cite{CLW}: $R>0$, $S<1$, and $\sup_M S=1$,
where $R$ and $S$ denote the scalar curvature and the squared norm of the second fundamental form, respectively.
For this case, it seems a hard task that the generalized maximum principle yields useful information. 
In order to overcome this substantial  difficulty, our  key ingredient  is to get a uniform positive lower bound 
for the Bakry-\'Emery Ricci curvature so that  we can make use of the comparison theorem of Wei--Wylie \cite{WeiWylie} 
to conclude  that  the Gaussian volume is finite.  This use of Bakry--\'Emery Ricci curvature provides a new approach to the classification of self-shrinkers. 
Furthermore, we establish a second gap theorem for $S$.
\end{abstract}
\maketitle

\section{Introduction and main results}

The analysis of singularities is a central problem in mean curvature flow. 
Huisken's monotonicity formula and his study of rescaled flows \cite{Huisken90,Huisken93} 
place self-shrinking solutions of the mean curvature flow at the center of this analysis, in particular in the study of type I singularities. 
The work of Colding and Minicozzi \cite{CM} relates the rigidity and stability of self-shrinkers to generic singularity models. 
Throughout this paper, every  hypersurface is understood to be a smooth  immersed and  connected hypersurface without boundary. 
An $n$-dimensional hypersurface $X:M^n\to\R^{n+1}$ is called {\it a self-shrinker} if 
\begin{equation}\label{eq:shrinker}
H+\langle X,N\rangle=0,
\end{equation}
where $H$ and $N$ denote the mean curvature and the unit normal vector of the hypersurface $X:M^n\to\R^{n+1}$, respectively.

It is well-known that the distinction between immersed and embedded self-shrinkers is substantial. Abresch and Langer \cite{AL} classified closed 
self-shrinking curves; the circle is the only embedded member of that class. Halldorsson \cite{Halldorsson} studied all self-similar curve-shortening 
solutions, including complete non-proper shrinking curves. In dimension two, Drugan \cite{Drugan} constructed an immersed non-embedded self-shrinking 
sphere. On the other hand,  Brendle \cite{Brendle} proved that a 2-dimensional compact embedded self-shrinker of genus zero is the round sphere.
In this paper, we  concern  immersed self-shrinkers  and do not initially impose any  conditions of  embedding or polynomial volume growth.

The condition of the polynomial volume growth is a powerful and strong geometric condition. 
The mean convex classification developed by Huisken \cite{Huisken90,Huisken93} and Colding--Minicozzi \cite{CM} 
shows that the hyperplane, the round spheres, 
and the generalized cylinders are the only complete embedded self-shrinkers
with polynomial volume growth. Furthermore,
Ding--Xin \cite{DXVolume} established Euclidean upper volume growth estimates for complete proper self-shrinkers. 
Cheng--Zhou \cite{ChengZhou} proved the equivalence of properness, polynomial volume growth, and finite Gaussian  
volume for complete self-shrinkers. 
Complementarily, Li--Wei \cite{LWVolume} obtained at least linear volume growth 
for complete non-compact proper self-shrinkers. 
According to the following simple example, 
the polynomial volume growth seems  a substantial condition for Gaussian integration by parts on a non-compact self-shrinker.
\begin{example}\label{rem:cover}
The  complete immersed self-shrinker
\[
 X:\R\times\R^{n-1}\longrightarrow\R^{n+1},
 \qquad X(t,z)=(\cos t,\sin t,z)
\]
has $S=1$ and constant $f_3$, and its induced metric is Euclidean.
Its volumes of intrinsic balls are $\omega_n r^n$, while the inverse image
of $S^1(1)\times\{0\}$ is noncompact. Thus it is a complete nonproper
immersion.
This example also distinguishes intrinsic volume growth from extrinsic volume growth
counted with multiplicity.
\end{example}
Curvature pinching provides another route in order to  classify complete self-shrinkers. Cao--Li \cite{CaoLi} obtained 
the first  gap theorem on complete self-shrinkers with 
 polynomial volume growth under $S\le1$ where $S$ denotes the squared norm of the second fundamental form. 
 Ding--Xin \cite{DXRigidity}, Xu--Xu \cite{xx17} 
 and Lei--Xu--Xu\cite{lxx20} obtained the second gap on the second fundamental form under the condition of  polynomial volume growth.
Ding--Xin \cite{DXRigidity} also  developed an  integral rigidity result on the second fundamental form under the condition of  polynomial volume growth.
Cheng--Peng \cite{ChengPeng} introduced a generalized maximum principle for $\mathcal L$ 
 on complete self-shrinkers with Ricci curvature bounded from below, permitting rigidity arguments without assumption on the polynomial volume growth. 
 Li--Wei \cite{LWRigidity} extended classification and rigidity results in higher codimension under additional geometric assumptions. 
 These works provide both the geometric models and the analytic methods.

The classification problem on complete self-shrinkers with constant scalar curvature  is  
often viewed as a self-shrinker counterpart of the Chern problem for compact minimal hypersurfaces in the unit sphere:
\vskip1mm
\noindent
{\bf Chern problems}.
For $n$-dimensional compact minimal  hypersurfaces in $S^{n+1}(1)$
with constant  scalar curvature,  is the following true?
\begin{enumerate}
\item $S\leq c(n)$, where $c(n)$ is a constant depending only on dimension $n$,
\item the values of $S$  are discrete, 
\item the values of $S$
should determine the hypersurfaces up to a rigid motion in the ambient sphere
$S^{n+1}(1)$,
\end{enumerate}
where $S$ denotes the squared norm  of the second fundamental  form.
\vskip1mm
\noindent
For complete self-shrinkers, the expected hypersurface models are $S^n(\sqrt n)$, $\R^n$ and $S^k(\sqrt{k})\times\R^{n-k}$, $1\le k\le n-1$. 
The classification problem on complete self-shrinkers with constant scalar curvature,  
as a self-shrinker counterpart of the Chern problem 
for compact minimal hypersurfaces in the unit sphere, will be  very important and interesting.
In \cite{Guo}, Guo proved that compact self-shrinkers with constant scalar curvature in $\mathbb R^{n+1}$
are  the sphere $S^{n}(\sqrt n)$. Since an $n$-dimensional compact self-shrinker  
in $\mathbb R^{n+1}$ must have a convex point, the  scalar curvature must be positive at this point.
By making use of Stokes formula, Guo \cite{Guo} proved  the scalar curvature $R= (n-1)$.  Thus, 
Chern type problems on compact self-shrinkers with constant scalar curvature were resolved by Guo \cite{Guo}. 
On the other hand,  the study on $n$-dimensional complete non-compact self-shrinkers in $\mathbb R^{n+1}$ is  
more important and  in this case, one can not use  integral formulas and Stokes theorem. 
This will be a substantial problem, which one should  overcome.
Luo, Sun, and Yin \cite{LSY} have classified $n$-dimensional complete self-shrinkers in $\mathbb R^{n+1}$ with
zero  scalar curvature.  
 Chern type problem on complete self-shrinkers with positive constant scalar curvature  is the following:
 
\vskip2mm
\noindent
{\bf Problem}. For an $n$-dimensional complete self-shrinker with positive constant scalar curvature,
its image, up to an orthogonal transformation, is one of  the round sphere $S^n(\sqrt n)$ or  the standard generalized cylinder
$S^k(\sqrt k)\times \R^{n-k}$,  for $2\leq k\leq n-1$?
\vskip2mm
For the above problem, 
Cheng--Li--Wei \cite{CLW} have made an important breakthrough 
by making use of the  generalized maximum principle of Cheng--Peng \cite{ChengPeng}. 
But it is a hard task that the generalized maximum principle yields useful information 
for the case that $S<1$ and $\sup_M S=1$.\\
In this paper, our main purpose is to resolve the above problem completely.
We address this issue through the Gaussian weighted geometry. The decisive ingredient is 
to get a uniform positive lower bound 
for the Bakry-\'Emery Ricci curvature  so that  we can make use of the comparison theorem of Wei--Wylie \cite{WeiWylie} 
to conclude  that the Gaussian volume is finite. According to our knowledge, it seems the first time to use 
the Bakry-\'Emery Ricci curvature to classify self-shrinkers.

\begin{theorem}\label{thm:main}
For an $n$-dimensional  complete self-shrinker $X:M^n\to\R^{n+1}$ with  positive  constant scalar curvature, 
its image $X(M^n)$, up to an orthogonal transformation,  is  one of
\begin{enumerate}
\item $S^n(\sqrt n)$,
\item
$S^k(\sqrt k)\times\R^{n-k}, \quad 2\le k\le n-1$.
\end{enumerate}
\end{theorem}

\begin{remark} In  our  theorem~\ref{thm:main}, we do not assume the polynomial volume growth or embeddedness,
which are strong geometric conditions  in view of Example 1.1.
We shall introduce the Bakry--\'Emery Ricci tensor $\Ric_{\varphi}$
associated with $\varphi=\frac{|X|^2}2$, which is used to overcome  the substantial  difficulty in 
Cheng--Li--Wei \cite{CLW}. Thus, we can exclude case : $R>0$, $S<1$, and $\sup_M S=1$,
which can not be done by the generalized maximum principle.
\end{remark}

As an obvious  consequence  of the theorem \ref{thm:main}, we have the following:

\begin{corollary}\label{cor:residual}
There are no complete  self-shrinkers  $X:M^n\to\R^{n+1}$  with positive constant scalar curvature and $S<1$.
\end{corollary}

On the other hand, for a minimal hypersurface in the unit sphere, 
the Gauss equation implies that constancy of scalar curvature is equivalent to constancy of the squared norm of the second fundamental form. 
But for a  self-shrinker, the Gauss equation $R=H^2-S$ does not give such an equivalent relation  because $H$ may not be  a constant in general. 
In place of constancy of scalar curvature,  as a counterpart of Chern problem for compact minimal hypersurfaces in the unit sphere, 
we would like to ask whether  for an $n$-dimensional complete self-shrinker with  constant squared norm of the second fundamental form,
its image, up to an orthogonal transformation, is one of  the round sphere $S^n(\sqrt n)$, $\R^n$ or  the standard generalized cylinder
$S^k(\sqrt k)\times \R^{n-k}$,  for $1\leq k\leq n-1$?
For this problem, Cheng--Ogata \cite{ChengOgata} resolved  the two-dimensional case, affirmatively.
In dimension three, Cheng, Z. Li, and Wei \cite{CLWf4, CLWf3} obtained classifications under additional
assumption  that  $f_3$ is constant or  $f_4$ is constant, respectively, where
$
f_3=\sum_i\lambda_i^3,\,  f_4=\sum_i\lambda_i^4,
$
and $\lambda_i$ denotes principal  curvatures.
Cheng--Wei \cite{ChengWei} established a gap theorem on complete self-shrinkers with constant norm of 
the second fundamental form and  with polynomial volume growth. Furthermore, Cheng--Wei--Yano \cite{CWY} 
obtained a second gap result under constancy of $S$ and $f_3$, without assuming polynomial volume growth. 
Recently, Cheng--Li--Wei \cite{CLW} have resolved this problem if the scalar curvature is bounded from below by $-\frac75$.

For the second gap on $S$,  Ding--Xin \cite{DXRigidity}, 
Xu--Xu \cite{xx17} and Lei--Xu--Xu \cite{lxx20}
proved that for an $n$-dimensional complete self-shrinker with polynomial volume growth, if $1\leq S\leq \dfrac{19}{18}$, then 
$S\equiv 1$. Since the condition of  polynomial volume growth forces  a strong geometric restriction for self-shrinkers, 
we will consider  the second gap on $S$ for complete self-shrinkers without the condition 
of polynomial volume growth. Here we need to remark that we do not assume $S$ is constant.

\begin{theorem}\label{thm:2gap}
For $n\ge2$,  let $X:M^n\to\R^{n+1}$ be an $n$-dimensional complete self-shrinker with scalar curvature $ R\ge r_0$,
where  $r_0>\dfrac{n-2}{18}+\dfrac{\sqrt{19(n-1)}}9$ is a constant.
If
\begin{equation}
1\leq S\le1+\dfrac {1}{18},
\end{equation}
then $S\equiv1$ and  its image $X(M^n)$, up to an orthogonal transformation,  is  one of
\begin{enumerate}
\item $S^n(\sqrt n)$,
\item
$S^k(\sqrt k)\times\R^{n-k}, \quad 2\le k\le n-1$.
\end{enumerate}
\end{theorem}

\section{Geometric conventions and identities}

We will follow the geometric notation in \cite[Section 2]{CLW}. The metric $g$ on a hypersurface  $X:M^n\to\R^{n+1}$
is the induced metric. 
Choose a local orthonormal frame $\{e_1, \cdots, e_n, e_{n+1}\}$ with $e_{n+1}=N$ and dual coframe $\{\omega_1, \cdots, \omega_n, \omega_{n+1}\}$. 
The second fundamental form and its coefficients are given by $\vec h$ and $h_{ij}$
\[
\omega_{i,n+1}=\sum_j h_{ij}\omega_j,\qquad h_{ij}=h_{ji},\qquad
\vec h=\sum_{i,j}h_{ij}\omega_i\otimes\omega_j\,e_{n+1}.
\]
We put
\[
H=\sum_i h_{ii},\qquad S=\sum_{i,j}h_{ij}^2=\sum_i\lambda_i^2,
\]
where $\lambda_i$'s  are the principal curvatures. $H$ is called the mean curvature of the hypersurface $X:M^n\to\R^{n+1}$.
The covariant derivatives $h_{ijk}$ of the second fundamental form are defined by
\[
\sum_k h_{ijk}\omega_k=dh_{ij}+\sum_kh_{ik}\omega_{kj}+\sum_kh_{kj}\omega_{ki}.
\]
We have the Codazzi equation: for any $i, \ j, \ k$, 
\[
\qquad h_{ijk}=h_{ikj}.
\]
A hypersurface $X:M^n\to\R^{n+1}$ is called {\it a self-shrinker} if $H+\langle X,N\rangle=0$.
The equation $H+\langle X,N\rangle=0$ does not depend on the choice of the unit normal vector  $N$.  Gauss equations are given by 
\begin{equation}\label{eq:gauss}
R_{ijkl}=h_{ik}h_{jl}-h_{il}h_{jk},\qquad
R_{ij}=Hh_{ij}-\sum_kh_{ik}h_{kj},\qquad R=H^2-S.
\end{equation}
In particular, $H^2\le nS$ implies
\begin{equation}\label{eq:Rbound}
R\le(n-1)S.
\end{equation}
Set
\begin{equation}\label{eq:weight}
\varphi=\frac{|X|^2}{2},\qquad
d\mu_{\varphi}=e^{-\varphi}d\mu,\qquad
\Delta_{\varphi} u=\Delta u-\langle\nabla \varphi,\nabla u\rangle.
\end{equation}
Since $\nabla\varphi=X^T$, we have $\mathcal L u=\Delta u-\langle X,\nabla u\rangle=\Delta_{\varphi}u$, 
exactly the operator in \cite[Section 2]{CLW}. The auxiliary potential $\varphi$ is introduced for 
the weighted argument. A direct differentiation gives
\begin{equation}\label{eq:se}
\begin{aligned}
&\varphi_i=\nabla_i\varphi=\langle X, e_i\rangle,\\
&\varphi_{ij}=\nabla_j\nabla_i\varphi=\delta_{ij}+\langle X,N\rangle h_{ij}=\delta_{ij}-Hh_{ij}.
\end{aligned}
\end{equation}
The Bakry-\'Emery Ricci tensor $\Ric_{\varphi}$ in this paper is defined by 
\begin{equation}\label{eq:ricf}
(\Ric_{\varphi})_{ij}:=R_{ij}+\varphi_{ij}.
\end{equation}
According to the Gauss equation and (\ref{eq:se}), we have
\begin{equation}
(\Ric_{\varphi})_{ij}=\delta_{ij}-\sum_k h_{ik}h_{kj}.
\end{equation}
The following integral is called Gaussian volume:
$$
\int_Me^{-\varphi}d\mu=\int_Md\mu_{\varphi}=\int_Me^{-\frac{|X|^2}2}d\mu.
$$
Since self-shrinker $X:M^n\to\R^{n+1}$ is complete, the Gaussian volume may be positive infinite.
In view of 
\[
H_i=\sum_kh_{ik}\langle X,e_k\rangle,
\quad
H_{ij}=\sum_kh_{ijk}\langle X,e_k\rangle+h_{ij}-H\sum_kh_{ik}h_{kj},
\]
and 
\[
\Delta h_{ij}=H_{ij}+H\sum_kh_{ik}h_{kj}-Sh_{ij},
\]
we have  Simons identity for self-shrinkers:

\begin{equation}\label{eq:simons}
\begin{aligned}
&\mathcal L h_{ij}=(1-S)h_{ij},\\
&\frac12\mathcal L S=\sum_{i,j,k}h_{ijk}^2+S(1-S),
\end{aligned}
\end{equation}
where the Codazzi equation is used. 

\section{A lower bound on the Bakry--\'Emery  Ricci curvature}

In this section, we will  give an estimate on the Bakry--\'Emery  Ricci curvature. 
In order to do it, we prove an algebraic  lemma.
\begin{lemma}\label{lem:algebra}
Let $\lambda_1,\ldots,\lambda_n\in\R$ be real numbers with $n\ge2$, such  that
\[
\sum_i\lambda_i^2=S,\qquad
\bigl (\sum_i\lambda_i\bigl)^2-\sum_i\lambda_i^2\ge r_0>0.
\]
Then, for every $j$,  $1\leq j\leq n$, we have
\begin{equation}\label{eq:delta}
\lambda_j^2\le S-\dfrac{r_0^2}{\biggl(\sqrt{(n-2)r_0+(n-1)S}+\sqrt{(n-1)S}\biggl)^2}.
\end{equation}
\end{lemma}

\begin{proof} Since $\bigl (\sum_i\lambda_i\bigl)^2-\sum_i\lambda_i^2\ge r_0>0$, we know $S=\sum_i\lambda_i^2>0$.
For a fixed $j$, defining  $t_j=\sqrt{\sum_{i\ne j}\lambda_i^2}$ and $b_j=\sum_{i\ne j}\lambda_i$, 
we have 
$$
0\le t_j\le \sqrt S, \quad  |\lambda_j|\le \sqrt S \quad  \text{\rm and} \quad   |b_j|\le\sqrt{n-1}\ t_j.
$$ Hence, we obtain 
\begin{align*}
r_0&\leq \bigl (\sum_i\lambda_i\bigl)^2-\sum_i\lambda_i^2\\
&= 2\lambda_j b_j+b_j^2-t_j^2\\
&\le2\sqrt{(n-1)S}\ t_j+(n-2)t_j^2.\\
\end{align*}
If $n=2$, we have 
$$
t_j\geq \dfrac{r_0}{2\sqrt{(n-1)S}}
$$
If $n>2$, we have 
$$
r_0+\dfrac{n-1}{n-2}S\leq \biggl (\sqrt{\dfrac{n-1}{n-2}S}+\sqrt {n-2}\ t_j\biggl )^2.
$$
Thus, we obtain
$$
t_j\geq \dfrac{r_0}{\sqrt{(n-2)r_0+(n-1)S}+\sqrt{(n-1)S}}.
$$
Hence, for $n\geq 2$, we conclude 
$$
\lambda_j^2=S-t_j^2\le S-\dfrac{r_0^2}{\biggl(\sqrt{(n-2)r_0+(n-1)S}+\sqrt{(n-1)S}\biggl)^2}.
$$
\end{proof}
\begin{remark}
The estimate concerns the largest value of $ \lambda_j^2$, not merely their sum $S$. 
For a complete self-shrinker with $S\leq a$ for some positive number $a$, 
even if $S$ approaches $a$, no individual $\lambda_j^2$ can approach $a$ as long as the scalar curvature  $R$ remains 
uniformly positive.  This distinction supplies the strict uniform bound for  the Bakry--\'Emery  Ricci curvature $\Ric_{\varphi}$ if $a=1$.
\end{remark}

\begin{proposition}\label{prop:ric}
For an $n$-dimensional self-shrinker $X:M^n\to\R^{n+1}$, if the scalar curvature $R$ satisfies  $R\geq r_0>0$, 
then the Bakry--\'Emery  Ricci tensor  $\Ric_{\varphi}$ satisfies 
$$
\Ric_{\varphi}\ge\bigl(1-S+\dfrac{r_0^2}{\biggl(\sqrt{(n-2)r_0+(n-1)S}+\sqrt{(n-1)S}\biggl)^2} \bigl)g
$$
 with $\varphi=\frac{|X|^2}2$.
\end{proposition}
\begin{proof}
We choose the orthonormal  frame $\{e_1, \cdots, e_n\}$ such that the second fundamental form is given by  $h_{ij}=\lambda_i\delta_{ij}$ at each point,
where  $\lambda_i$, for $1\leq i\leq n$, denotes  the principal curvature of the self-shrinker $X:M^n\to\R^{n+1}$. We have 
$H=\sum_i\lambda_i$, $S=\sum_i\lambda_i^2$ and $H^2-S=R\geq r_0>0$.
At each point, 
$$
\bigl(\sum_i\lambda_i\bigl)^2-\sum_i\lambda_i^2=H^2-S=R\geq r_0>0.
$$
According to the lemma \ref{lem:algebra}, we get that the principal curvature $\lambda_j$  for any $j$ satisfies
 $$
\lambda_j^2\le S-\dfrac{r_0^2}{\biggl(\sqrt{(n-2)r_0+(n-1)S}+\sqrt{(n-1)S}\biggl)^2}.
$$
In view of  $\varphi=\frac{|X|^2}2$, we have
$$
\nabla_i\varphi=\langle X, e_i\rangle, \quad \nabla_j\nabla_i\varphi=\delta_{ij}+\langle X, N\rangle h_{ij}
$$
The Bakry--\'Emery  Ricci tensor is given by, in view of the Gauss equation and the above formula, 
$$
(\Ric_{\varphi})_{ij}=R_{ij}+\nabla_j\nabla_i\varphi=\delta_{ij}-\lambda_i^2\delta_{ij}.
$$
Hence, we obtain 
$$
(\Ric_{\varphi})_{ij}\geq \bigl(1-S+\dfrac{r_0^2}{\biggl(\sqrt{(n-2)r_0+(n-1)S}+\sqrt{(n-1)S}\bigl)^2}\biggl)\delta_{ij}.
$$
This finishes  the proof of the proposition \ref{prop:ric}.
\end{proof}

\section{Proofs of main theorems}

In this section, we will give  proofs of our theorems. 
\begin{theorem}\label{thm:}
For $n\ge2$,  let $X:M^n\to\R^{n+1}$ be an $n$-dimensional complete self-shrinker.
If
\begin{equation}
S\le1\quad\text{and}\quad R\ge r_0>0
\end{equation}
for some positive constant $r_0$, then $S\equiv1$ and  
its image $X(M^n)$  is, up to an orthogonal transformation, one of $S^n(\sqrt n)$ or  $S^k(\sqrt{k})\times\R^{n-k}$ for some
$k\in\{2,\ldots,n-1\}$.
\end{theorem}

In order to prove Theorem \ref{thm:}, we first prove the following proposition.
\begin{proposition}\label{prop:growth}
Let $X:M^n\to\R^{n+1}$ be an $n$-dimensional  complete self-shrinker as in the theorem~\ref{thm:}. we have 
\[
\int_M e^{-|X|^2/2}\dd\mu<\infty,
\]
that  is, the Gaussian volume is finite.
\end{proposition}

\begin{proof}
In view of the proposition~\ref{prop:ric}, we know that the Bakry-\'Emery Ricci curvature satisfies
$$
\begin{aligned}
(\Ric_{\varphi})_{ij}&\geq \bigl(1-S+\dfrac{r_0^2}{\biggl(\sqrt{(n-2)r_0+(n-1)S}+\sqrt{(n-1)S}\biggl)^2}\bigl)\delta_{ij}\\
&\geq \bigl(\dfrac{r_0^2}{\biggl(\sqrt{(n-2)r_0+(n-1)}+\sqrt{(n-1)}\biggl)^2}\bigl)\delta_{ij}
\end{aligned}
$$
because of $S\leq 1$.
Hence, the Bakry-\'Emery Ricci tensor  satisfies $\Ric_{\varphi}\ge\delta g$ for a  constant $\delta$,
where 
$$\delta=\bigl(\dfrac{r_0^2}{\biggl(\sqrt{(n-2)r_0+(n-1)}+\sqrt{(n-1)}\biggl)^2}\bigl)>0.
$$
The comparison  theorem of Wei--Wylie \cite[Corollary~5.1]{WeiWylie} therefore gives
\[
\mu_{\varphi}(M)=\int_M e^{-|X|^2/2}\dd\mu
<\infty.
\]
Thus, the Gaussian volume is finite.
\end{proof}

\vskip2mm
\noindent
\begin{proof}[Proof of Theorem~\ref{thm:}] 
By $Y=\sqrt 2 X$, $Y:M^n\to\R^{n+1}$ satisfies $\bar H+\frac12\langle Y, N\rangle=0$, where $\bar H$ is the mean curvature
of  $Y:M^n\to\R^{n+1}$. Thus,  $Y:M^n\to\R^{n+1}$ is a self-shrinker in the setting  of Cheng--Zhou \cite{ChengZhou}. Since
the scaling does not change completeness,  $Y:M^n\to\R^{n+1}$ is complete.
According to Theorem in Cheng--Zhou \cite{ChengZhou},
we know that for  $Y:M^n\to\R^{n+1}$, finite Gaussian volume, properness, Euclidean volume growth
and polynomial volume growth are equivalent. From Proposition  \ref{prop:growth}, we know
\[
\int_M e^{-|Y|^2/4}\dd\bar{\mu}
=2^{\frac n2}\int_M e^{-|X|^2/2}\dd\mu<\infty.
\]
Hence,  $Y:M^n\to\R^{n+1}$  is proper and has polynomial volume growth.  
Since the scaling does not change properness and the volume 
is only changed by a constant multiple, which only depends on dimension $n$, 
we conclude that the self-shrinker $X:M^n\to\R^{n+1}$ 
has  polynomial volume growth. In fact, Cheng and Wei \cite{CW1} have given a direct proof for  $\lambda$-hypersurfaces,
which is a natural generalization of self-shrinkers in our setting.
According to Theorem of Cao--Li \cite[Theorem~1.1]{CaoLi}, we obtain 
 $S\equiv 1$ and its  image $X(M^n)$ is, up to an orthogonal transformation, one of $S^n(\sqrt n)$ or  $S^k(\sqrt{k})\times\R^{n-k}$ for some
$k\in\{2,\ldots,n-1\}$ since the scalar curvature is  positive.
\end{proof}

\begin{proof}[Proof of Theorem~\ref{thm:main}] In view of  Theorem in Cheng--Li--Wei \cite{CLW}, for a complete 
self-shrinker $X:M^n\to\R^{n+1}$ with positive constant scalar curvature, up to an orthogonal transformation, 
its image $X(M^n)$ is one of $S^n(\sqrt n)$ or  $S^k(\sqrt{k})\times\R^{n-k}$ for some
$k\in\{2,\ldots,n-1\}$
or $S<1$, $\sup S=1$  and $R>0$. 
If $S<1$, since the scalar curvature $R$ is a positive constant, 
then conditions in Theorem~\ref{thm:}  are satisfied. Hence, 
Theorem~\ref{thm:} gives that there are  no this kind of complete self-shrinkers. 
Hence,  this completes the proof of  the theorem \ref{thm:main}.
\end{proof}

\begin{proof}[Proof of Theorem~\ref{thm:2gap}]
According  the proposition~\ref{prop:ric},  the Bakry-\'Emery Ricci curvature satisfies
$$
\begin{aligned}
(\Ric_{\varphi})_{ij}&\geq \bigl(1-S+\dfrac{r_0^2}{\biggl(\sqrt{(n-2)r_0+(n-1)S}+\sqrt{(n-1)S}\bigl)^2}\biggl)\delta_{ij}.\\
\end{aligned}
$$
Because of  $S\leq 1+\frac1{18}$, we infer
$$
\begin{aligned}
(\Ric_{\varphi})_{ij}
&\geq \bigl( -\dfrac1{18}+ \dfrac{r_0^2}{\biggl(\sqrt{(n-2)r_0+\dfrac{19}{18}(n-1)}+\sqrt{\dfrac{19}{18}(n-1)}\biggl)^2}\bigl)\delta_{ij}.
\end{aligned}
$$
In view of $r_0>\dfrac{n-2}{18}+\dfrac{\sqrt{19(n-1)}}9$,  we have
$$
\begin{aligned}
&(r_0-\dfrac{n-2}{18}-\dfrac{\sqrt{19(n-1)}}9)r_0>0,\\
&r_0^2-\dfrac{2}{\sqrt {18}}\sqrt{\dfrac{19}{18}(n-1)}r_0+\dfrac{\frac{19}{18}(n-1)}{18}>\dfrac1{18}\bigl\{(n-2)r_0+\frac{19}{18}(n-1)\bigl\},\\
& \dfrac{r_0^2}{\biggl(\sqrt{(n-2)r_0+\dfrac{19}{18}(n-1)}+\sqrt{\dfrac{19}{18}(n-1)}\biggl)^2}>\dfrac1{18}.\\
\end{aligned}
$$
Hence, we obtain
$$
\delta= -\dfrac1{18}+ \dfrac{r_0^2}{\biggl(\sqrt{(n-2)r_0+\dfrac{19}{18}(n-1)}+\sqrt{\dfrac{19}{18}(n-1)}\biggl)^2}>0.
$$
Thus,  the Bakry-\'Emery Ricci tensor  satisfies $\Ric_{\varphi}\ge\delta g$ for the  positive constant $\delta$.
The comparison theorem of  Wei--Wylie \cite[Corollary~5.1]{WeiWylie} gives that the Gaussian volume satisfies
\[
\mu_{\varphi}(M)=\int_M e^{-\varphi}\dd\mu
=\int_M e^{-|X|^2/2}\dd\mu<\infty.
\]
Thus, the Gaussian volume is finite. According to the same assertion as in Theorem \ref{thm:}, we know that 
 the self-shrinker  $X:M^n\to\R^{n+1}$  has  polynomial volume growth. Conditions in Theorem of 
Lei--Xu--Xu \cite[Theorem~1.1]{lxx20}  are satisfied. Hence, we know  $S\equiv 1$. Therefore, from  Theorem of Cao--Li \cite[Theorem~1.1]{CaoLi}, 
we obtain  that its image $X(M^n)$, up to an orthogonal transformation,  is  one of
 $S^n(\sqrt n)$,
or 
$S^k(\sqrt k)\times\R^{n-k}, \quad 2\le k\le n-1$ since the scalar curvature is positive.
\end{proof}
 \vskip2mm
\noindent
{\bf Acknowledgement}.
This work was partly supported by MEXT Promotion of Distinctive Joint Research Center Program JPMXP0723833165 
and Osaka Metropolitan University Strategic Research Promotion Project (Development of International Research Hubs).


\begin{thebibliography}{99}
\bibitem{AL}
U. Abresch and J. Langer,
\emph{The normalized curve shortening flow and homothetic solutions},
J. Differential Geom. \textbf{23} (1986), 175--196.

\bibitem{Brendle}
S. Brendle,
\emph{Embedded self-similar shrinkers of genus $0$},
Ann. of Math. (2) \textbf{183} (2016), 715--728.

\bibitem{CaoLi}
H.-D. Cao and H. Li,
\emph{A gap theorem for self-shrinkers of the mean curvature flow in arbitrary codimension},
Calc. Var. Partial Differential Equations \textbf{46} (2013), 879--889.

\bibitem{CLW}
Q.-M. Cheng, F. Li, and G. Wei,
\emph{Estimates on scalar curvature of self-shrinkers},
J. Reine  Angew. Math.,  DOI 10.1515/crelle-2026-0091, 2026.

\bibitem{CLWf4}
Q.-M. Cheng, Z. Li, and G. Wei,
\emph{Complete self-shrinkers with constant norm of the second fundamental form},
Math. Z. \textbf{300} (2022), 995--1018.

\bibitem{CLWf3}
Q.-M. Cheng, Z. Li, and G. Wei,
\emph{Classification of complete $3$-dimensional self-shrinkers in the Euclidean space $\R^4$},
Sci. China Math. \textbf{67} (2024), 873--882.


\bibitem{ChengOgata}
Q.-M. Cheng and S. Ogata,
\emph{$2$-dimensional complete self-shrinkers in $\R^3$},
Math. Z. \textbf{284} (2016), 537--542.

\bibitem{ChengPeng}
Q.-M. Cheng and Y. Peng,
\emph{Complete self-shrinkers of the mean curvature flow},
Calc. Var. Partial Differential Equations \textbf{52} (2015), 497--506.


\bibitem{ChengWei}
Q.-M. Cheng and G. Wei,
\emph{A gap theorem of self-shrinkers},
Trans. Amer. Math. Soc. \textbf{367} (2015), 4895--4915.

\bibitem{CW1}
Q.-M. Cheng and G. Wei, 
\emph{Stability and area growth of $\lambda$-hypersurfaces}, Comm. 
Analy. Geom.,  \textbf{30} (2022), 1059--1091.

\bibitem{CWY}
Q.-M. Cheng, G. Wei, and W. Yano,
\emph{The second gap on complete self-shrinkers},
Proc. Amer. Math. Soc. \textbf{151} (2023), 339--348.


\bibitem{ChengZhou}
X. Cheng and D. Zhou,
\emph{Volume estimate about shrinkers},
Proc. Amer. Math. Soc. \textbf{141} (2013), 687--696.


\bibitem{CM}
T. H. Colding and W. P. Minicozzi II,
\emph{Generic mean curvature flow I: Generic singularities},
Ann. of Math. (2) \textbf{175} (2012), 755--833.

\bibitem{DXVolume}
Q. Ding and Y. L. Xin,
\emph{Volume growth, eigenvalue and compactness for self-shrinkers},
Asian J. Math. \textbf{17} (2013), 443--456.


\bibitem{DXRigidity}
Q. Ding and Y. L. Xin,
\emph{The rigidity theorems of self-shrinkers},
Trans. Amer. Math. Soc. \textbf{366} (2014), 5067--5085.

\bibitem{Drugan}
G. Drugan,
\emph{An immersed $S^2$ self-shrinker},
Trans. Amer. Math. Soc. \textbf{367} (2015), 3139--3159.


\bibitem{Guo}
Z. Guo,
\emph{Scalar curvature of self-shrinkers},
J. Math. Soc. Japan \textbf{70} (2018), 1103--1110.

\bibitem{Halldorsson}
H. P. Halldorsson,
\emph{Self-similar solutions to the curve shortening flow},
Trans. Amer. Math. Soc. \textbf{364} (2012), 5285--5309.


\bibitem{Huisken90}
G. Huisken,
\emph{Asymptotic behavior for singularities of the mean curvature flow},
J. Differential Geom. \textbf{31} (1990), 285--299.

\bibitem{Huisken93}
G. Huisken,
\emph{Local and global behaviour of hypersurfaces moving by mean curvature},
in \emph{Differential Geometry: Partial Differential Equations on Manifolds}
(Los Angeles, CA, 1990), Proc. Sympos. Pure Math. \textbf{54}, Part 1,
Amer. Math. Soc., Providence, RI, 1993, 175--191.

\bibitem{LWVolume}
H. Li and Y. Wei,
\emph{Lower volume growth estimates for self-shrinkers of mean curvature flow},
Proc. Amer. Math. Soc. \textbf{142} (2014), 3237--3248.

\bibitem{LWRigidity}
H. Li and Y. Wei,
\emph{Classification and rigidity of self-shrinkers in the mean curvature flow},
J. Math. Soc. Japan \textbf{66} (2014), 709--734.

\bibitem{LSY}
Y. Luo, L. Sun, and J. Yin,
\emph{Complete self-similar hypersurfaces to the mean curvature flow with nonnegative constant scalar curvature},
Front. Math. \textbf{18} (2023), 417--430.

\bibitem{WeiWylie}
G. Wei and W. Wylie,
\emph{Comparison geometry for the Bakry--\'Emery Ricci tensor},
J. Differential Geom. \textbf{83} (2009), 377--405.

\bibitem{xx17}
H. W. Xu and Z. Y. Xu, On Chern's conjecture for minimal hypersurfaces and rigidity of
self-shrinkers, J. Funct. Anal., 273(2017), 3406-3425.
\bibitem{lxx20}
L. Lei, H. W. Xu, and Z. Y. Xu, A new pinching theorem for complete self-shrinkers and its generalization,
 Sci. China Math. 63 (2020), 1139-1152. 
 \end{thebibliography}
\end{document}